\documentclass[reqno,a4paper]{amsart} 
\usepackage{amsmath}    % math enhancements latex2e (replaces amstex)
\usepackage{amsthm}     % proclaims theoremstyle/proof environments
\usepackage{amssymb}    % AMSFonts and symbols
\usepackage{amsfonts}
\usepackage{color}

\newtheorem{theorem}{Theorem}[section]

\newtheorem{lemma}[theorem]{Lemma}
\newtheorem{corollary}[theorem]{Corollary}

\theoremstyle{remark}

\numberwithin{equation}{section}

\newcommand{\abs}[1]{\lvert#1\rvert}
\newcommand{\norm}[1]{\|#1\|}

\newcommand{\sprod}[2]{\langle #1, #2 \rangle}
\def\d{\partial}

\def\R{{\mathbb R}}
\def\eti{\tilde{e}^i}
\def\E{\mathcal{E}}
\def\Et{\tilde{\mathcal{E}}}
\def\Rt{\tilde{R}}
\def\Kt{\tilde{K}}
\def\Qt{\tilde{Q}_T}

\def\ds{\displaystyle}
\begin{document}
\title[A posteriori error estimates for Vlasov-Maxwell system]{A posteriori error estimates for the 
one and one-half Dimensional Relativistic Vlasov-Maxwell system
 }
\author[M.~Asadzadeh]{Mohammad Asadzadeh$^\dag$}

\address{
Department of Mathematics,
Chalmers University of Technology and  G\"oteborg University,
SE--412 96, G\"oteborg, Sweden
}

\email{mohammad@chalmers.se}

\author[C.~Standar]{Christoffer Standar$^\ddag$}

% \address{
% Department of Mathematics,
% Chalmers University of Technology and  G\"oteborg University,
% SE--412 96, G\"oteborg, Sweden
% }
\email{standarc@chalmers.se}

\keywords{Streamline Diffusion, Vlasov-Maxwell, 
a posteriori error estimates, stability, convergence
}

\subjclass{ 65M15, 65M60}

\footnote{$^\ddag$ Corresponding author}  
\thanks{$^\dag$ The research of this author was partially supported by the Swedish Research Council (VR)}

\begin{abstract}
This paper concerns a posteriori error analysis for the 
streamline diffusion (SD)
finite element method for the one and one-half dimensional 
relativistic Vlasov-Maxwell system. The SD scheme yields a weak formulation,
that corresponds to an add of extra diffusion to, e.g. 
the system of equations having hyperbolic nature, and 
 convection-dominated convection diffusion problems. A procedure that 
improves the convergence of finite elements for this type of problems. 
The a posteriori error estimates relay on a dual problem formulation and 
yields an error control based on the, computable, residual of the 
approximate solution. The lack of dissipativity  
enforces us considering {\sl negative norm estimates}. 
To derive these estimates, the  error term is split 
into an iteration  
and an approximation error where the iteration procedure is assumed to   
converge.
The computational aspects and implementations, 
which justify the theoretical results of this part,  
 are the subject of these studies and addressed in \cite{Bondestam_Standar:2017}. 

\end{abstract}

\maketitle

\section{Introduction} \label{sec1}
This paper concerns a posteriori error analysis 
for approximate solution of the Vlasov-Maxwell (VM) system  by the
{\sl streamline diffusion} (SD) finite element methods.
Our main objective is to prove a posteriori error estimates for 
the SD scheme in the 
$H_{-1}(H_{-1})$ and $L_\infty(H_{-1})$ norms for the Maxwell equations 
and 
 $L_\infty(H_{-1})$ norm for the Vlasov part. The VM system 
lacking dissipativity exhibits severe stability draw-backs and 
 the usual $L_2(L_2)$ and 
$L_\infty(L_2)$ errors are only bounded by the residual norms. 
Thus, in order not to rely on the smallness of the residual errors,  
we employ the negative norm estimates 
to pick up convergence rates also involving powers of the mesh parameter $h$ 
and having optimality properties due to the maximal available regularity 
of the exact solution. 
Both Vlasov and Maxwell equations are of 
hyperbolic type and for the exact solution in the Sobolev space $H^{r+1}$, 
the classical finite element method for 
hyperbolic partial differential equations 
will have, an optimal, convergence rate 
of order 
${\mathcal O}(h^r)$, where $h$ is the mesh size. 
On the other hand, with the same regularity 
($H^{r+1}$) the optimal convergence rate 
for the elliptic and parabolic problems is of order 
${\mathcal O}(h^{r+1})$. This phenomenon, and the lack of diffusivity 
in the hyperbolic equations which cause oscillatory behavior 
in the finite element schemes, sought for constructing modified 
finite element schemes that could enhance stability and improve 
the convergence behavior for hyperbolic problems. In this regard, 
compared to the classical 
finite element, the SD 
schemes, corresponding to the add of diffusion term to the hyperbolic equation, 
are more stable and have an improved convergence rate viz, 
${\mathcal O}(h^{r+1/2})$. Roughly, 
the SD method is based on 
a weak formulation where a multiple of convection term 
is added to the test function. With this choice of the test 
functions the variational formulation resembles to that of an equation which, 
compared to the original hyperbolic equation, has an additional diffusion term 
of the order of the multiplier. 

A difficulty arises deriving gradient 
estimates for the dual problems, which are crucial for the 
error analysis for the discrete models in 
 both equation types in the VM system. This is due to the lack of 
dissipative terms in the equations. An elaborate discussion 
on this issue can be found in the classical results, e.g.,  
 \cite{Lax-Phillips:1960}, \cite{Friedrichs} and \cite{Tartakoff} as well as in 
relatively recent studies in \cite{Johnson} and \cite{Suli-Houston:97}.

We use the advantage of low spatial dimension that, assuming 
sufficient regularity, yields existence and uniqueness through d'Alembert 
formula. 
This study 
can be extended to  higher dimensional geometries, where  
a different analytical approach for the well-posedness is available in 
the studies by Glassey and Schaeffer in,
e.g.,  \cite{Glassey_Schaeffer:2001} and \cite{Glassey_Schaeffer:2000}. 
Numerical 
implementations for this model will appear in the second part: 
\cite{Bondestam_Standar:2017}. We also mention related studies 
\cite{Monk_Sun:2012} and  \cite{Perugia_Schotzau_Monk:2002}
for the Maxwell's equations  
 where stabilized interior penalty method is used.

Problems of this type have been considered by several 
authors in various settings. In this regard, 
theoretical studies for the Vlasov-Maxwell system relevant to our work 
can be found in, e.g.\ \cite{Diperna_Lions:89} for 
treating the global weak solutions, \cite{Yan_Guo:93} 
for global weak solutions with 
boundary conditions and more adequately 
\cite{Gla}-\cite{Glassey_Schaeffer:2000} 
for relativistic models in different geometries. SD methods for
 the hyperbolic partial differential equations have been suggested by 
Hughes and Brooks in \cite{Hughes_Brooks:79}. Mathematical developments can be 
found in \cite{John}. For SD studies relevant to our approach 
see, e.g.,\  \cite{Asadzadeh_Kowalczyk:2005} and the references therein  
some containing also further studies involving discontinuous Galerkin schemes 
and their developments.

An outline of this paper is as follows: In the present Section \ref{sec1}, following 
the introduction, we comment on particular manner of various quantities in the 
Maxwell equations and introduce the relativistic one and one-half 
dimensional model with its well-posedness property. 
In Section \ref{sec2} we introduce some notations and preliminaries.
Sections \ref{sec3} is devoted to stability 
bounds and a posteriori error estimates for
the Maxwell equations in both $H_{-1}(H_{-1})$ and 
$L_\infty(H_{-1})$ norms. 
Sections \ref{sec4} is the counterpart of Section 3 for the Vlasov 
equation which is now performed only in $L_\infty(H_{-1})$ norm.  

Finally, in our concluding Section \ref{sec5}, 
we summarize the results of the paper and discuss some future plans. 

Throughout this note $C$ will denote a generic constant, 
not necessarily the same at each occurrence, and independent of the parameters 
in the equations, unless otherwise explicitly specified. 

The Vlasov-Maxwell (VM) system which describes time evolution of collisionless 
plasma is formulated as 
%\begin{display}
\begin{equation}\label{VMI}
\begin{aligned}
\partial_t f &+\hat v\cdot \nabla_x f+q(E+c^{-1} \hat v\times B)\cdot \nabla_v f=0,\\
\mbox{(Ampere's law)}\quad  &\partial_t E=c \nabla\times B-j,\qquad \nabla\cdot E=\rho, \\
\mbox{(Faraday's law)}\quad  &\partial_t B=-c\nabla\times E, \qquad  \nabla\cdot B=0.
\end{aligned}
\end{equation}
Here $f$ is density, in phase space, of particles with mass $m$, 
charge $q$ and  velocity 
$$
\hat v=(m^2+c^{-2}\vert v\vert^2)^{-1/2} v\qquad (v\,\,  \mbox{is momentum}).
$$
Further, the charge and current densities are given by 
\begin{equation}\label{VMII}
\rho(t,x)=4\pi\int qf\, dv,\quad\mbox{and}\quad 
 j(t,x)= 4\pi\int qf\hat v\, dv, 
\end{equation}
respectively. For a proof of the existence and uniqueness of the solution to 
VM system one may rely on mimicking the Cauchy problem for the 
Vlasov equation through using Schauder fixed point theorem:  
{\sl Insert an assumed and  given   $g$ for $f$ in \eqref{VMII}. 
Compute $\rho_g$, $j_g$ and insert the results in 
Maxwell equations to get $E_g$, $B_g$.  Then 
insert, such obtained, $E_g$ and $B_g$ in the Vlasov equation 
to get $f_g$ via an operator $\Lambda$:  $f_g=\Lambda g$. 
A fixed point of $\Lambda$ is the solution of the Vlasov equation. 
For the discretized version employ, instead, the  Brouwer 
fixed point theorem.
Both these proofs are rather technical and non-trivial.} 
The fixed points argument, 
rely on viewing the equations in the Maxwell's system as being valid 
independent of each others, %which is not a physically correct view.
but the quantities $f$, $B$, $E$, $j$ and $\rho$ are {\sl physically} 
related to each others by 
the Vlasov-Maxwell system of equations and % again, {\sl physically}, 
it is {\sl not the case} 
that some of them are given to determine the others.  
However, in one and one-half geometry, relying on d'Alembert formula 
Schauder/Brouwer fixed point approach, is unnecessary.   
The fixed point approach, which was first introduced by Ukai and Okabe in 
\cite{Ukai_Okabe:78} 
for the Vlasov-Poisson system, is 
performed for the Vlasov-Maxwell system in 
\cite{Standar:2015} in full details and therefore is omitted in here. 

\subsection
{Relativistic model in one and one-half dimensional geometry}
Our objective is to construct and analyze SD discretization schemes 
for the 
{\sl relativistic Vlasov-Maxwell
model in one and one-half dimensional geometry} 
($x\in{\mathbb R}, v\in {\mathbb R}^2$), which then can 
be generalized to higher dimensions: 
\begin{equation} \label{rvm}
\mbox{(RVM)} \qquad\left\{
\begin{aligned}
&\partial_t f +\hat v_1\cdot \partial_x f+ (E_1 + \hat{v}_2 B)\d_{v_1} f + (E_2-\hat{v}_1 B)\d_{v_2} f=0,\\
&\partial_t E_1=- j_1 (t, x) ,\qquad \partial_x E_1=\rho (t, x) = \int_{\Omega_v} f dv - \rho_b (x) , \\
&\partial_t E_2 + \partial_x B=- j_2 (t, x) , \\
&\partial_t B + \partial_x E_2= 0 .
\end{aligned}
\right.
\end{equation}
The system \eqref{rvm} is assigned with the Cauchy data 
\[
f(0,x,v)=f^0(x,v) \geq 0, \,\,\,\, E_2(0,x)=E_2^0(x), \,\,\,\, B(0,x)=B^0(x)
\]
and with 
\[
E_1 (0,x)= \int_{-\infty}^x \Big( \int f^0 (y,v) dv- \rho_b (y)\Big) dy = E_1^0 (x).
\]
This is the only initial data that leads to a finite-energy solution (see \cite{Gla}). In \eqref{rvm}
we have for simplicity set all constants equal to one.
The background density $ \rho_b (x) $ is assumed to be smooth, has compact support and 
is {\sl neutralizing}. This yields 
\begin{equation*}
{\int_{-\infty}^{\infty}\rho(0,x)\, dx=0}. 
\end{equation*} 

To carry out the discrete analysis, we shall need the following 
global existence of 
classical solution due to Glassey and Schaeffer \cite{Gla}.
\begin{theorem}[Glassey, Schaeffer]
Assume that $ \rho_b $, the background density, is neutralizing and we have  
$$
(i) \,\, 0\le f^0(x,v)\in C_0^1({\mathbb R}^3),\qquad 
(ii) \,\, E_2^0,\,\,\, B^0 \in C_0^2({\mathbb R}^1).
$$
Then, there exists a {\sl global} $C^1$ solution for the Relativistic 
Vlasov-Maxwell system. Moreover, if 
$ 0\le f^0\in C_0^r({\mathbb R}^3)$ and 
$E_2^0,\,\,\, B^0 \in C_0^{r+1}({\mathbb R}^1)$, then $(f, E, B)$ is of 
class $C^r$ over ${\mathbb R}^+\times{\mathbb R}\times{\mathbb  R^2}$. 
\end{theorem}  
Note that for the well-posedness of the discrete solution
the existence and uniqueness is due to \cite{Standar:2015}, whereas 
the stability of the approximation scheme is justified 
throughout Sections 3 and 4. 

\section{Assumptions and notations} \label{sec2}

Let $ \Omega_x \subset \R $ and $ \Omega_v \subset \R^2 $ denote the 
space and velocity domains, respectively.
We shall assume that $ f(t,x,v) $ , $ E_2(t,x) $, $ B(t,x) $ and $ \rho_b (x) $ 
have compact supports in $ \Omega_x $ and that $ f(t,x,v) $ 
has compact support in $ \Omega_v $. 
Since we have assumed neutralizing background density, i.e.\ 
$ \int \rho (0,x) dx=0 $, 
it follows that $ E_1 $ also has compact support in $ \Omega_x $ 
(see \cite{Gla}). 

Now we will introduce a finite element structure on 
$ \Omega_x \times \Omega_v $. Let $ T_h^x = \{ \tau_x \} $ 
and $ T_h^v = \{ \tau_v \} $ be finite elements subdivision of 
$ \Omega_x $ with elements $ \tau_x $ and $ \Omega_v $ 
with elements $ \tau_v $, respectively. 
Then $ T_h = T_h^x \times T_h^v = \{ \tau_x \times \tau_v \} = \{ \tau \} $ 
is a subdivision of 
$ \Omega_x \times \Omega_v $. 
Let $ 0= t_0 < t_1 < \ldots < t_{M-1} < t_M=T $ 
be a partition of $ [ 0, T ] $ into sub-intervals 
$ I_m = (t_{m-1}, t_m ] $, $ m= 1, 2, \ldots , M $. 
Further let $ \mathcal{C}_h $ be the corresponding subdivision of 
$ Q_T = [0 ,T ] \times \Omega_x \times \Omega_v $ 
into elements $ K = I_m \times \tau $, with $ h= \textnormal{diam} \, K $ 
as the mesh parameter. Introduce $ \mathcal{\tilde{C}}_h $ 
as the finite element subdivision of 
$ \Qt = [0 , T ] \times \Omega_x $. 
Before we define our finite dimensional spaces we need to introduce
 some function spaces, viz 
\[
\mathcal{H}_0 = 
\prod_{m=1}^{M} H_0^1 (I_m \times \Omega_x \times \Omega_v)
\quad \mbox{and} \quad \mathcal{\tilde{H}}_0 = 
\prod_{m=1}^{M} H^1_0 ( I_m \times \Omega_x ),
\]
where
\[
H^1_0 (I_m \times \Omega) = 
\{ w \in H^1 ; w= 0 \, \, \textnormal{on} \,\, \d \Omega \}.
\]
In the discretization part, for 
 $ k=0,1,2, \ldots $, we define the finite element spaces
\[
V_h= \{ w \in \mathcal{H}_0 ; w|_K \in P_k (I_m ) \times P_k (\tau_x) \times P_k ( \tau_v ) , \, \forall K = I_m \times \tau \in \mathcal{C}_h \}
\]
and
\[
\tilde{V}_h = \{ g \in [\mathcal{\tilde{H}}_0]^3 ; g_i|_{\Kt} \in P_k (I_m) \times P_k (\tau_x ) , \, \forall \Kt = I_m \times \tau_x \in \mathcal{\tilde{C}}_h , \, i=1,2,3\},
\]
where $ P_k ( \cdot ) $ is the set of polynomial with degree at most $ k $ on the given set. We shall also use some notation, viz
\[
(f,g)_m=(f,g)_{S_m}, \qquad \| g \|_m =(g,g)_m^{1/2}
\]
and
\[
\sprod{f}{g}_m=(f(t_m, \ldots), g(t_m, \ldots))_{\Omega}, \qquad |g|_m=\sprod{g}{g}_m^{1/2},
\]
where $ S_m = I_m \times \Omega $, is the slab at $ m $-th level, $
 m=1, 2, \ldots, M $.

To proceed, we shall need to 
perform an iterative procedure: starting with $f^{h,0} $ we compute the fields 
$E_{1}^{h,1},  E_2^{h,1} $ and $ B^{h,1} $ and insert them in the 
Vlasov equation 
to get the numerical approximation $ f^{h,1} $. This will then be 
inserted in the Maxwell equations to get
the fields $E_{1}^{h,2},  E_2^{h,2} $ and $ B^{h,2} $ and so on. 
The iteration step $ i $ yields a Vlasov equation for $ f^{h,i } $ 
with the fields 
$E_{1}^{h,i},  E_2^{h,i} $ and $ B^{h,i }$. We are going to assume 
that this iterative procedure 
converges to the analytic solution of the Vlasov-Maxwell system. 
More specifically, we have assumed that the 
iteration procedure generates Cauchy sequences. 

Finally, due to the lack of dissipativity, 
we shall consider {\sl negative norm estimates}. Below we introduce 
the general form of the function spaces that will be useful 
in stability studies and supply us the adequate environment to derive 
error estimates with higher convergence rates. In this regard: 
Let $\Omega$ be a bounded domain in $\mathbb R^N$, $N \ge 2$.  For 
$m \ge 0$ an integer, $1 \le p \le \infty$ and $G \subseteq \Omega$,
$W^m_p(G)$ denotes the usual Sobolev space of functions with 
distributional 
derivatives of order $\le m$ which are in $L_p(G)$.  Define the 
seminorms
$$
|u|_{W^j_p(G)}
= \left\{ \begin{array}{ll} 
\Big(\ds\sum_{|\alpha|=j} \|D^\alpha u\|^p_{L_p(G)}
\Big)^{1/p}  &\mbox{if} \quad 1 \le p < \infty ,\\
\ds\sum_{|\alpha|=j} \|D^\alpha u\|_{L_\infty(G)}
&\mbox{if} \quad p = \infty,
\end{array}
\right .
$$
and the norms
$$
\|u\|_{W^m_p(G)}
= \left\{ \begin{array}{ll} 
\Big(\ds\sum^m_{j=1} |u|^p_{W_p^j(G)}\Big)^{1/p}
&\text{if} \quad 1 \le p < \infty, \\
\ds\sum^m_{j=1} |u|_{W^j_\infty(G)} &\mbox{if} \quad p = \infty.
\end{array}
\right .
$$
If $m\ge 0$, $W_p^{-m}(G)$ is the completion of 
$C^\infty_0(G)$ under the norm
$$
\|u\|_{W_p^{-m}(G)}
= \sup_{  \begin{array}{ll} 
& \psi\in C^\infty_0(G) % \\ & \|\psi\|_{W^m_q(G)}=1
\end{array}}
\frac{(u, \psi)}{ \|\psi\|_{W^m_q(G)}}, 
%\int_G u\psi dx,\
 \quad  \frac 1p + \frac 1q = 1.
$$
We shall only use the $L_2$-version of the above norm. 
\section{A posteriori error estimates for the Maxwell equations} \label{sec3}

Our main goal in this section is to find an a 
posteriori error estimate for the Maxwell equations. 
Let us first reformulate the relativistic Maxwell system, viz 
\begin{equation} \label{rm}
\left\{ \begin{array}{ll}
\displaystyle \d_x E_1 =\int f dv - \rho_b (x)= \rho(t,x) \\
\displaystyle \d_t E_1 = - \int \hat{v}_1 f dv= - j_1(t,x) \\
\displaystyle \d_t E_2 +\d_x B = -\int \hat{v}_2 f dv=- j_2(t,x) \\
\displaystyle \d_t B + \d_x E_2 =0 .
\end{array} \right.
\end{equation}

Set now 
\[
M_1= \left( \begin{array}{cccc}
0 & 0 & 0 \\
1 & 0 & 0 \\
0 & 1 & 0 \\
0 & 0 & 1 \\
\end{array} \right) , \,\,\,\, 
M_2= \left( \begin{array}{cccc}
1 & 0 & 0 \\
0 & 0 & 0 \\
0 & 0 & 1 \\
0 & 1 & 0 \\
\end{array} \right) .
\]
Let $ W= (E_1, E_2, B)^T $, $ W^0=(E_1^0, E_2^0, B^0) $ and 
$ b=(\rho, - j_1, - j_2, 0)^T $. Then, the Maxwell equations can be 
written in compact (matrix equations) form as 
\begin{equation} \label{maxwell}
\left\{ \begin{array}{ll}
M_1 W_t+ M_2 W_x=b \\
W(0,x)=W^0(x).
\end{array} \right. 
\end{equation} 

The streamline diffusion method on the $ i $th step for the 
Maxwell equations can now be formulated as: find $ W^{h,i} \in \tilde{V}_h $ 
such that for $ m=1, 2, \ldots , M $,
\begin{multline} \label{sdmax}
(M_1 W^{h,i}_t + M_2 W^{h,i}_x, \hat{g}+ \delta (M_1 g_t + M_2 g_x))_m + 
\sprod{W^{h,i}_+}{g_+}_{m-1} \\
= (b^{h,i-1}, \hat{g} + \delta (M_1 g_t + M_2 g_x))_m + 
\sprod{W^{h,i}_-}{g_+}_{m-1}, \,\,\,\,\,\, \forall \, g\in \tilde{V}_h,
\end{multline}
where $ \hat{g} = (g_1, g_1, g_2, g_3)^T $, $ g_\pm (t,x) = 
\lim_{s \rightarrow 0^\pm} g(t+s, x) $ and $ \delta $ is a multiple of $ h $ 
(or a multiple of $ h^\alpha $ for some suitable $ \alpha $), 
see \cite{Johnson} for motivation of choosing $ \delta $.

Now we are ready to start the a posteriori error analysis. 
Let us decompose the error into two parts
\[
W-W^{h,i}= \underbrace{W-W^i}_{\textnormal{analytical iteration error}} + 
\underbrace{W^i-W^{h,i}}_{\textnormal{numerical error}}= \Et^i+\eti ,
\]
where $ W^i $ is the exact solution to the approximated Maxwell 
equations at the $ i $th iteration step: 
\[
M_1 W^i_t+ M_2 W^i_x=b^{h,i-1} .
\]

\subsection{$H^{-1} (H^{-1})$ a posteriori error analysis for the Maxwell equations} 
We will start by estimating the numerical error $\eti$. 
To this end, we formulate the dual problem:
\begin{equation} \label{Dual_maxwell}
\left\{ \begin{array}{l}
-M_1^T \hat\varphi_t- M_2^T \hat\varphi_x= \chi \\
\varphi (T,x)=0.
\end{array} 
\right. 
\end{equation} 
Here $ \chi $ is a function in $ [H^1 ( \tilde{Q}_T )]^3 $. The idea is to use the dual problem
to get an estimate on the $ H^{-1} $-norm of the error $ \eti $.
Multiplying \eqref{Dual_maxwell} with $ \eti $ and integrating over $ \tilde{Q}_T $
we obtain
\begin{equation}\label{DualPFormula1}
(\eti,\chi)=
\sum_{m=1}^M \Big((\eti, -M_1^T\hat\varphi_t)_m + (\eti, -M_2^T\hat\varphi_x)_m \Big), 
\end{equation} 
where 
\begin{equation}\label{DualPFormula2}
\begin{split}
(\eti, -M_1^T\hat\varphi_t)_m = &
-\int_{S_m} \eti\cdot\partial_t(\varphi_1, \varphi_2, \varphi_3)\, dx\, dt\\
 =&
-\int_{S_m} (\eti_1\partial_t\varphi_1+ 
\eti_2\partial_t\varphi_2+\eti_3\partial_t\varphi_3) 
\, dx\, dt \\
 = & -\int_{\Omega_x}
\Big[\sum_{k=1}^3 \eti_k\varphi_k\Big]_{t=t_{m-1}}^{t=t_m}\, dx \\
& +\int_{S_m} (\partial_t \eti_1\varphi_1+\partial_t \eti_2\varphi_2
+\partial_t \eti_3\varphi_3)
\, dx\, dt \\
 =&\langle \eti_+, \varphi_+\rangle_{m-1}-
\langle \eti_-, \varphi_-\rangle_{m}+
(M_1 \eti_t, \hat \varphi)_m. 
\end{split}
\end{equation} 
Likewise, due to the fact that all involved functions have compact support in 
$\Omega_x$, we can write 
\begin{equation}\label{DualPFormula3}
\begin{split}
(\eti, -M_2^T\hat\varphi_x)_m &= 
-\int_{S_m} \eti \cdot\partial_x(\varphi_1, \varphi_2, \varphi_3)\, dx\, dt \\
&=
-\int_{S_m} (\eti_1\partial_x\varphi_1+ \eti_2\partial_x\varphi_2+\eti_3\partial_x\varphi_3) 
\, dx\, dt \\
&=\int_{S_m} (\partial_x \eti_1 \varphi_1+\partial_x \eti_2 \varphi_2+\partial_x \eti_3 \varphi_3) 
\, dx\, dt 
 =
(M_2 \eti_x, \hat\varphi)_m. 
\end{split}
\end{equation} 
Inserting \eqref{DualPFormula2} and \eqref{DualPFormula3} into the 
error norm  \eqref{DualPFormula1}, we get 
\begin{equation*} 
\begin{split}
(\eti,\chi) =&
\sum_{m=1}^M  \langle \eti_+ ,\varphi_+\rangle _{m-1} 
-\langle \eti_-, \varphi_-\rangle_m +(M_1 \eti_t+M_2 \eti_x, \hat\varphi)_m \\
=&\sum_{m=1}^M \langle \eti_+ - \eti_- + \eti_- ,\varphi_+\rangle _{m-1} 
-\langle \eti_-, \varphi_- -\varphi_+ + \varphi_+ \rangle_m \\
&+(M_1 W^i_t+M_2 W^i_x - M_1 W_t^{h,i} - M_2 W_x^{h,i}, \hat\varphi)_m \\
=& \sum_{m=1}^M \langle \eti_- ,\varphi_+ \rangle _{m-1} -
\langle \eti_- , \varphi_+ \rangle_m + \langle [ \eti ] , \varphi_+\rangle_{m-1}
+\langle \eti_-, [\varphi ]\rangle_m\\
&+\Big(b^{h,i-1}- M_1 W_t^{h,i} - M_2 W_x^{h,i} , \hat\varphi \Big)_m.
\end{split}
\end{equation*}
Now since both $\varphi$ and $W$ are continuous we have that 
$[\varphi ]=[W]\equiv 0$ and hence $ [\eti ] =- [W^{h,i} ]$. Thus 
\begin{equation} \label{DualPFormula4}
\begin{split}
(\eti,\chi)=&\langle \eti_- , \varphi_+\rangle_0 - \langle \eti_-, \varphi_+ \rangle_M
+ \sum_{m=1}^M - \langle [W^{h,i} ],\varphi_+\rangle_{m-1}\\
&+ 
\Big(b^{h,i-1}- M_1 W_t^{h,i} - M_2 W_x^{h,i} , \hat\varphi \Big)_m \\
=& \sum_{m=1}^M - \langle [W^{h,i} ],\varphi_+\rangle_{m-1}+ 
\Big(b^{h,i-1}- M_1 W_t^{h,i} - M_2 W_x^{h,i} , \hat\varphi \Big)_m . 
\end{split}
\end{equation} 
Let now $ \tilde{\varphi} $ be an interpolant of $ \varphi $ 
and use \eqref{sdmax} with $ g = \tilde{\varphi} $ to get
\begin{equation} \label{DualPFormula5}
\begin{split}
(\eti,\chi) =& \sum_{m=1}^M  \langle [W^{h,i} ], \tilde{\varphi}_+ - \varphi_+\rangle_{m-1}\\
&+ 
\Big(b^{h,i-1}- M_1 W_t^{h,i} - M_2 W_x^{h,i} , \hat\varphi  - \hat{\tilde{\varphi}} - \delta ( M_1 \tilde{\varphi}_t + M_2 \tilde{\varphi}_x) \Big)_m .
\end{split}
\end{equation}
Now, to proceed we introduce the residuals
$$
\Rt_1^i=b^{h,i-1} - M_1 W_t^{h,i} - M_2 W_x^{h,i}
$$
and
$$
\Rt^i_2 |_{S_m} = \Big( W_+^{h,i} (t_m, x) - W_-^{h,i} (t_m, x) \Big)/h, 
$$
where the latter one is constant in time on each slab.

Further, we shall use two projections, $ P $ and $ \pi $, 
for our interpolants $ \tilde{\varphi} $. 
These projections will be constructed from the local projections
\[
P_m : [L_2 (S_m) ]^3 \rightarrow \tilde{V}_m^h = 
\{ u|_{S_m} ; u \in \tilde{V}^h \}
\]
and
\[
\pi_m : [ L_2 (S_m) ]^3 \rightarrow \Pi_{0,m} = \{ u \in [ L_2 (S_m) ]^3 ; \,
u ( \cdot , x) \,\, \textnormal{is constant on} \,\, I_m, x \in \Omega_x \} ,
\]
defined such that
$$
\int_{\Omega_x}(P_m\varphi)^T\cdot u\, dx=
\int_{\Omega_x}\varphi^T\cdot u\, dx,\qquad \forall u\in \tilde{V}_m^h 
$$
and 
$$
\pi_mu|_{S_m}=\frac{1}{h} \int_{I_m} u(t, \cdot)\, dt .
$$
Now we define $ P $ and $ \pi $, slab-wise, by the formulas 
$$
(P\varphi)|_{S_m}=P_m(\varphi|_{S_m}) \quad \mbox{and} 
\quad (\pi\varphi)|_{S_m}=\pi_m(\varphi|_{S_m}), 
$$
respectively. See Brezzi et al.\ \cite{Brezzi} 
for the details on commuting  differential and 
projection operators in a general setting. 
Now we may choose the interpolants as 
$\tilde\varphi=P \pi \varphi = \pi P \varphi $, 
and write an error representation formula as 
\begin{equation}\label{DualPFormula6}
\begin{split} 
\sum_{m=1}^M&
\Big\langle [W^{h,i}], \varphi_+ - \tilde\varphi_+\Big \rangle_{m-1} = 
\sum_{m=1}^M \Big\langle h\frac{[W^{h,i}]}h, \varphi_+-P \varphi_+ 
+ P \varphi_+ - \tilde\varphi_+ \Big\rangle_{m-1} \\
&= \sum_{m=1}^M \Big\langle h\frac{[W^{h,i}]}h, \varphi_+ 
- P \varphi_+ \Big\rangle_{m-1}
+\sum_{m=1}^M \Big\langle h \frac{[W^{h,i}]}h , P \varphi_+ - 
\tilde\varphi_+\Big\rangle_{m-1}\\
& := J_1+J_2.
\end{split}
\end{equation} 
To estimate $ J_1 $ and $ J_2 $ 
we shall use the following identity
\[
h \varphi_+ (t_{m-1}, x) = \int_{I_m} \varphi (t, \cdot ) dt - \int_{I_m} \int_{t_{m-1}}^t \varphi_s (s, \cdot ) ds dt .
\]
 We estimate each term in the error representation 
formula separately:
\begin{equation}\label{DualPFormula7}
\begin{split} 
J_1 = &\sum_{m=1}^M \Big\langle h\frac{[W^{h,i}]} h, 
\varphi_+-P \varphi_+\Big\rangle_{m-1} \\
= &\sum_{m=1}^M \Big\langle\frac{W_+^{h,i} - 
W_-^{h,i} }{h}, (I-P) h \varphi_+\Big\rangle_{m-1} \\
=& \sum_{m=1}^M \Big\langle \Rt^i_2 , (I-P) 
\Big( \int_{I_m} \varphi (t, \cdot )\, dt -
\int_{I_m} \int_{t_{m-1}}^t \varphi_s(s, \cdot)\, ds\, dt\Big)
\Big\rangle_{m-1}\\
= & \sum_{m=1}^M \Big(\int_{\Omega_x} \Rt_2^i \cdot 
(I-P)\int_{I_m} \varphi (t,x )\, dt\, dx\Big)\\
&- \sum_{m=1}^M \Big(\int_{\Omega_x} \Rt_2^i \cdot 
(I-P)\int_{I_m} \int_{t_{m-1}}^t \varphi_s(s, \cdot)\, ds\, dt \Big)\\
&\le C\norm{h \Rt^i_2}_{L_2 ( \Qt )} \norm{\varphi}_{ H^1 ( \Qt )}, 
\end{split}
\end{equation}
where in the last estimate the, piecewise time-constant, 
 residual is moved inside the 
time integration. 
As for the $J_2$-term we have that 
\begin{equation}\label{DualPFormula8}
\begin{split} 
\sum_{m=1}^M & \Big\langle h\frac{[W^h]}h, 
 P \varphi_+ - \tilde{\varphi}_+ \Big\rangle_{m-1}
=\sum_{m=1}^M \Big\langle \Rt^i_2, 
P h \varphi_+ - h\tilde\varphi_+\Big\rangle_{m-1} \\
&= \sum_{m=1}^M \Big\langle \Rt_2^i, \int_{I_m} P \varphi (t, \cdot)\, dt- 
\int_{I_m} \int_{t_{m-1}}^t P \varphi_s(s, \cdot)\, ds\, dt - 
h \tilde\varphi_+\Big\rangle_{m-1} \\
&= - \sum_{m=1}^M \int_{I_m} \int_{t_{m-1}}^t \langle \Rt^i_2 , P \varphi_s(s, \cdot)
\rangle_{m-1} \, ds\, dt.  
\end{split}
\end{equation} 
Thus we can derive the estimate  
\begin{equation}\label{DualPFormula9} 
\abs{J_2} \le 
C \norm{h \Rt^i_2}_{L_2 ( \Qt )} \norm{P \varphi_t}_{L_2 ( \Qt )}
\le C \norm{h \Rt^i_2}_{L_2 ( \Qt )} \norm{\varphi}_{H^1 ( \Qt )}. 
\end{equation} 
To estimate the second term in \eqref{DualPFormula5} 
we proceed in the following way
\begin{equation}\label{DualPFormula10}
\begin{split}
\sum_{m=1}^M & \Big(b^{h,i-1}-  M_1 W_t^{h,i} - M_2 W_x^{h,i} , \hat\varphi - \hat{\tilde{\varphi}} - \delta ( M_1 \tilde{\varphi}_t + M_2 \tilde{\varphi}_x) \Big)_m \\
= &\sum_{m=1}^M \Big( \Rt_1^i,  \hat{\varphi} - \hat{\tilde{\varphi}} \Big)_m 
- \delta \Big( \Rt_1^i , M_1 \varphi_t + M_2 \varphi_x \Big)_m \\
& + \delta \Big( \Rt_1^i, M_1 (\varphi_t- \tilde{\varphi}_t) 
+ M_2 (\varphi_x - \tilde{\varphi}_x)  \Big)_m \\
\leq & C \norm{\Rt_1^i}_{L_2 ( \Qt )} \norm{\varphi - \tilde{\varphi}}_{L_2 ( \Qt )}
+ C h \norm{\Rt_1^i}_{L_2 ( \Qt )}  \norm{\chi}_{L_2 ( \Qt )} \\
& + C h \norm{\Rt^i_1}_{L_2 ( \Qt )} \big( \norm{\varphi_t - \tilde{\varphi}_t}_{L_2 ( \Qt )} 
+ \norm{\varphi_x - \tilde{\varphi}_x}_{L_2 ( \Qt )} \big) \\
\leq & C h \norm{\Rt_1^i}_{L_2 ( \Qt )} \norm{\varphi}_{H^1 ( \Qt )} 
+ C h \norm{\Rt_1^i}_{L_2 ( \Qt )} \norm{\chi}_{H^1 ( \Qt )}. 
\end{split}
\end{equation} 
Combining \eqref{DualPFormula5}-\eqref{DualPFormula10} yields
\begin{equation*}
\begin{split}
( \eti, \chi ) \leq & C h \norm{\Rt_1^i}_{L_2 ( \Qt )} \norm{\varphi}_{H^1 ( \Qt )}
+ C h \norm{\Rt_1^i}_{L_2 ( \Qt )} \norm{\chi}_{H^1 ( \Qt )} \\
& + C h \norm{\Rt_2^i}_{L_2 ( \Qt )} \norm{\varphi}_{H^1 ( \Qt )}.
\end{split}
\end{equation*}
To get an estimate for the $ H^{-1} $-norm we need to divide both sides
by $ \norm{\chi}_{H^1 ( \Qt )} $ and take the supremum over 
$ \chi \in [H^1 ( \tilde{Q}_T )]^3 $.
We also need the following stability estimate. 
\begin{lemma} \label{stabmax}
There exists a constant $ C $ such that
\[
\norm{\varphi}_{H^1 ( \Qt )} \leq C \norm{\chi}_{H^1 ( \Qt )}.
\]
\end{lemma}
\begin{proof}
To estimate the $ H^1 $-norm of $ \varphi $ we first write out the equations for the dual problem explicitly:
\begin{equation} \label{DualPex}
\left\{ \begin{array}{l}
- \d_t \varphi_1 - \d_x \varphi_1 = \chi_1 \\
- \d_t \varphi_2 - \d_x \varphi_3 = \chi_2 \\
- \d_t \varphi_3 - \d_x \varphi_2 = \chi_3 . \\
\end{array} 
\right. 
\end{equation}
We start by estimating the $ L_2 $-norm of $ \varphi $.
Multiply the first equation by $ \varphi_1 $ and integrate over $ \Omega_x $ to get
\[
- \int_{\Omega_x} \d_t \varphi_1 \varphi_1 dx - \int_{\Omega_x} \d_x \varphi_1 \varphi_1 dx = \int_{\Omega_x} \chi_1 \varphi_1 dx .
\]
Standard manipulations yields
\[
- \frac{1}{2} \int_{\Omega_x} \d_t (\varphi_1)^2 dx - \frac{1}{2} \int_{\Omega_x} \d_x (\varphi_1)^2 dx \leq \norm{\chi_1}_{L_2 (\Omega_x )} \norm{\varphi_1}_{L_2 ( \Omega_x )} .
\]
The second integral vanishes because $ \varphi_1 $ is zero on the boundary of $ \Omega_x $. We therefore have the following inequality
\[
- \d_t \norm{\varphi_1}_{L_2 ( \Omega_x )}^2 \leq \norm{\chi_1}^2_{L_2 (\Omega_x )} + \norm{\varphi_1}_{L_2 ( \Omega_x )}^2 .
\]
Integrate over $ (t, T ) $ to get
\[
\norm{ \varphi_1 (t, \cdot )}^2_{L_2 (\Omega_x )} \leq 
\norm{\chi_1 }^2_{L_2 ( \Qt )} + \int_t^T \norm{ \varphi_1 (s, \cdot )}^2_{L_2 (\Omega_x )} ds .
\]
Applying Gr\"{o}nwall's inequality and then integrating over $ (0, T) $ we end up with the stability estimate
\[
\norm{\varphi_1 }_{L_2 ( \Qt )} \leq \sqrt{T} e^{T/2} \norm{\chi_1}_{L_2 ( \Qt )} .
\]
Similarly we estimate the second and third component of $ \varphi $ 
as follows: We 
multiply the second and the third equations of \eqref{DualPex} 
with $ \varphi_2 $ and $ \varphi_3 $, respectively. 
Adding the resulting equations and integrating over $ \Omega_x $, yields 
the equation
\begin{equation}\label{Stab4.2}
- \int_{\Omega_x} \varphi_2 \d_t \varphi_2  + 
\varphi_2  \d_x \varphi_3 + \varphi_3 \d_t \varphi_3  + 
\varphi_3 \d_x \varphi_2  dx = 
\int_{\Omega_x} \chi_2 \varphi_2 + \chi_3 \varphi_3 dx .
\end{equation}
We may rewrite \eqref{Stab4.2} as
\begin{equation}\label{Stab4.3}
\begin{split}
- \frac{1}{2} \int_{\Omega _x} \d_t (\varphi_2 )^2 + \d_t (\varphi_3 )^2 + 2 \d_x (\varphi_2 \varphi_3 ) dx \leq & 
\norm{\chi_2}_{L_2 (\Omega_x )} \norm{\varphi_2}_{L_2 ( \Omega_x )} \\
&+ \norm{\chi_3}_{L_2 (\Omega_x )} \norm{\varphi_3}_{L_2 ( \Omega_x )}.
\end{split}
\end{equation}
Note that the third term on the left hand side of \eqref{Stab4.3} 
is identically equal to zero because both $ \varphi_2 $ and $ \varphi_3 $ 
vanish at the boundary of $ \Omega_x $. We therefore have the following inequality
\[
- \d_t \left( \norm{\varphi_2}^2_{L_2 (\Omega_x )} + \norm{\varphi_3}^2_{L_2 (\Omega_x )} \right) \leq \norm{\chi_2}^2_{L_2 (\Omega_x )} +  \norm{\varphi_2}^2_{L_2 ( \Omega_x )} + \norm{\chi_3}^2_{L_2 (\Omega_x )} + \norm{\varphi_3}^2_{L_2 ( \Omega_x )} .
\]
Integrating over $ (t, T) $ we get that   
\begin{equation*}
\begin{split}
\norm{\varphi_2 (t, \cdot) }^2_{L_2 (\Omega_x )} 
+ \norm{\varphi_3 (t, \cdot )}^2_{L_2 (\Omega_x )} \leq &
\norm{\chi_2}^2_{L_2 ( \Qt )} + \norm{\chi_3 }^2_{L_2 ( \Qt )} \\
&+  \int_t^T \norm{\varphi_2 (s,\cdot) }^2_{L_2 ( \Omega_x )}  
+ \norm{\varphi_3 (s,\cdot) }^2_{L_2 ( \Omega_x )} ds .
\end{split}
\end{equation*} 
Applying Gr\"{o}nwall's inequality and then integrating over 
$ (0, T ) $ we end up with the stability estimate
\[
\norm{\varphi_2}_{L_2 ( \Qt )} + \norm{\varphi_3}_{L_2 ( \Qt )} 
\leq \sqrt{T} e^{T/2} ( \norm{\chi_2}_{L_2 ( \Qt )} 
+ \norm{\chi_3}_{L_2 ( \Qt )} ) .
\]
Next we need to prove that the $ L_2 $-norms of the derivatives of 
$ \varphi $ are bounded by
$ \norm{\chi}_{H^1 ( \Qt )} $. To do this we first note that 
$ \varphi $ has analytical solutions, see \cite{Standar:2015}, 
\begin{equation}
\begin{split}
\varphi_1 (t, x) = & \int_t^T \chi_1 (s, x + s - t) ds, \\
\varphi_2 (t, x) = & \frac{1}{2} \int_t^T \chi_2 (s, x + s - t) 
+ \chi_3 (s, x + s - t) \\
& \qquad + \chi_2 (s, x + t - s) - \chi_3 (s, x + t - s) \, ds, \\
 \varphi_3 (t, x) = & \frac{1}{2} \int_t^T \chi_2 (s, x + s - t) 
+ \chi_3 (s, x + s - t) \\
& \qquad - \chi_2 (s, x + t - s) + \chi_3 (s, x + t - s) \, ds.
\end{split}
\end{equation}
Let us start by estimating the $ x $-derivative of $ \varphi_1 $. 
By the above formula for
$ \varphi_1 $ we have that
\[
\frac{\d \varphi_1}{\d x} (t, x) = \int_t^T \frac{\d \chi_1}{\d x} (s, x + s - t) \, ds.
\]
Cauchy-Schwartz inequality and a suitable change of variables yields
\[
\begin{split}
\int_{\Omega_x} \left( \frac{\d \varphi_1}{\d x} (t, x) \right)^2 dx \leq &
T \int_{\Omega_x} \int_t^T \left| 
\frac{\d \chi_1}{\d x} (s, x + s - t) \right|^2 ds \, dx \\
\leq & T \int_0^T \int_{\Omega_x} 
\left| \frac{\d \chi_1}{\d x} (s, y) \right|^2 dy \, ds.
\end{split}
\]
Integrating both sides of the inequality over $ (0, T) $, gives the estimate
\[
\norm{\d_x \varphi_1}_{L_2 ( \Qt )} \leq T \norm{\d_x \chi_1}_{L_2 ( \Qt )}.
\]
Now we can use this inequality together with the first equation in 
\eqref{DualPex}
to get an estimate for the time derivative of $ \varphi_1 $:
\[
\norm{\d_t \varphi_1}_{L_2 ( \Qt )} 
\leq T \norm{\d_x \chi_1}_{L_2 ( \Qt )}
+ \norm{\chi_1}_{L_2 ( \Qt )}.
\]
Similar estimates can be derived for the derivatives of $ \varphi_2 $ and 
$ \varphi_3 $.
We omit the details and refer to 
the estimations of the derivatives for $ \varphi_1 $.
\end{proof} 
Summing up we have proved following estimate 
for the numerical error $ \eti $.
\begin{theorem}[A posteriori error] \label{ApostMaxHH}
There exists a constant $ C $ such that
\[
\norm{\eti}_{H^{-1} ( \Qt )} \leq C \left( \norm{h \Rt_1^i}_{L_2 ( \Qt )} + \norm{h \Rt_2^i}_{L_2 ( \Qt )} \right).
\]
\end{theorem} 
As for the iterative error $ \Et^i $ we assume that $ W^i $ 
converges to the analytic solution,  
so that, for sufficiently large $i$, 
$\eti$ is the dominating part of the error $ W - W^{h, i} $, 
see \cite{Standar:2015} for motivation of the iteration assumption.
Therefore, for large enough $i$,
we have that 
\begin{equation}\label{AApriori}
\norm{\Et^i }_{H^{-1} ( \Qt )} <<  \norm{\eti}_{H^{-1} ( \Qt )}. 
\end{equation} 
This together with Theorem \ref{ApostMaxHH} yields the following result: 
\begin{corollary} 
There exists a constant $ C $ such that
\[
\norm{W- W^{h,i}}_{H^{-1} ( \Qt )} \leq C \Big(
\norm{h \Rt^i_1}_{L_2 ( \Qt )} + \norm{h \Rt_2^i}_{L_2 ( \Qt )} \Big) .
\]
\end{corollary}

\subsection{$L_\infty(H^{-1})$ a posteriori error analysis for the 
Maxwell equations} 
In this part we perform a $ L_\infty (H^{-1} ) $ error estimate. 
The interest in this norm is partially due to the fact that the Vlasov
part is studied in the same environment.
To proceed we formulate a new dual problem as
\begin{equation} \label{DualMaxII}
\left\{ \begin{array}{l}
-M_1^T \hat\varphi_t- M_2^T \hat\varphi_x= 0 \\
\varphi (T,x)= \chi (x) ,
\end{array} 
\right. 
\end{equation}
where $ \chi \in [H^1 (\Omega_x)]^3 $.
We multiply $ \eti (T, x) $ by $ \chi $ and integrate over $ \Omega_x $ to get
\[
\sprod{\eti}{\chi}_M = \sprod{\eti_-}{\varphi_+}_M + ( \eti , - M_1^T \hat{\varphi}_t - M_2^T \hat{\varphi}_x ) .
\]
Using \eqref{DualPFormula2} and \eqref{DualPFormula3} the above identity can be written as
\[
\begin{split}
\sprod{\eti}{\chi}_M = & \sprod{\eti_-}{\varphi_+}_M \\ 
& + \sum_{m=1}^M \Big(\langle \eti_+ ,\varphi_+\rangle _{m-1} 
-\langle \eti_-, \varphi_-\rangle_m +(M_1 \eti_t+M_2 \eti_x, \hat\varphi)_m\Big).
\end{split}
\]
With similar manipulations as in \eqref{DualPFormula4}, this equation simplifies to
\[
\sprod{\eti}{\chi}_M = \sum_{m=1}^M - \langle [W^{h,i} ],\varphi_+\rangle_{m-1}+ 
\Big(b^{h,i-1}- M_1 W_t^{h,i} - M_2 W_x^{h,i} , \hat\varphi \Big)_m .
\]
Following the proof of Theorem \ref{ApostMaxHH} we end up with the following result:
\begin{theorem}
There exists a constant $ C $ such that
\begin{equation*}
\norm{\eti (T,\cdot )}_{H^{-1} (\Omega_x )} 
\le C \Big( \norm{h \Rt^i_1}_{L_2 ( \Qt )}
+ \norm{h \Rt_2^i}_{L_2 ( \Qt )} \Big).
\end{equation*}
\end{theorem}
In the proof of this theorem we use the stability estimate:
\begin{lemma}
There exists a constant $ C $ such that
\[
\norm{\varphi}_{H^1 ( \Qt )} \leq C \norm{\chi}_{H^1 (\Omega_x )}.
\]
\end{lemma}
The proof is similar to that of Lemma \ref{stabmax} and therefore is omitted.
With the same assumption on the iteration error $ \Et^i $ as in \eqref{AApriori},
the numerical error $ \eti $ will be dominant and we have the following final result:
\begin{corollary}
There exists a constant $ C $ such that
\begin{equation*}
\norm{W (T,\cdot ) - W^{h, i} (T,\cdot )}_{H^{-1} (\Omega_x )} 
\le C \Big( \norm{h \Rt^i_1}_{L_2 ( \Qt )}
+ \norm{h \Rt_2^i}_{L_2 ( \Qt )} \Big).
\end{equation*}
\end{corollary}

\section{A posteriori error estimates for the Vlasov equation} \label{sec4} 

The study of the Vlasov part rely on a gradient estimate for the dual solution. 
Here, the $L_2$-norm estimates, would only yield 
error bounds depending 
on the size of residuals, with no $h^\alpha$-rates. 
Despite the smallness of the residual norms this, however, does not imply 
concrete convergence rate and smaller residual norms require unrealistically 
finer degree of resolution. The remedy is to employ negative norm 
estimates, in order to gain convergence rates of the order
$h^\alpha $, for some $ \alpha>0.$ In this setting a $H_{-1}(H_{-1})$-norm 
is inappropriate. Hence, this section is devoted to 
$L_\infty(H_{-1})$-norm error estimates for the Vlasov equation in the 
Vlasov-Maxwell system. 

\subsection{ $L_\infty(H_{-1})$ a posteriori error estimates for the Vlasov equation}

The streamline diffusion method on the $ i $th step for the 
Vlasov equation can be formulated as: find $ f^{h,i} \in V_h $ such that for $ m=1, 2, \dots , M $,
\begin{equation} \label{sdvlasov}
\begin{split}
( f_t^{h,i} + G(f^{h,i-1}) \cdot \nabla f^{h,i} , g + & \delta (g_t + G(f^{h,i-1}) \cdot \nabla g ))_m \\
&+ \sprod{f^{h,i}_+}{g_+}_{m-1} = \sprod{f^{h,i}_-}{g_+}_{m-1}, \quad \forall 
g \in V_h,
\end{split}
\end{equation}
where the drift factor
\[
G(f^{h,i-1}) = ( \hat{v}_1 , E_1^{h,i} + \hat{v}_2 B^{h,i} , E_2^{h,i} - \hat{v}_1 B^{h,i} )
\]
is computed using the solutions of the Maxwell equations.
As in the Maxwell part we decompose the error into two parts
\[
f-f^{h,i}= \underbrace{f-f^i}_{\textnormal{analytical iteration error}} + \underbrace{f^i-f^{h,i}}_{\textnormal{numerical error}}= \E^i + e^i ,
\]
where $ f^i $ is the exact solution of the approximated Vlasov equation at the $ i $th iteration step: 
\begin{equation}\label{Exact_Vlasov1}
f^i_t + G( f^{h,i-1} ) \cdot \nabla f^i = 0.
\end{equation}
To estimate the numerical error we 
formulate a corresponding dual problem as
\begin{equation} \label{Dual_vlasov}
\left\{ \begin{array}{l}
-\Psi^i_t - G (f^{h,i-1} ) \cdot \nabla \Psi^i = 0, \\
\Psi^i(T,x,v)=\chi (x, v),
\end{array} 
\right. 
\end{equation}
where $ \chi \in H^1 ( \Omega_x \times \Omega_v ) $.
Multiplying $ e^i (T, x, v) $ by $ \chi $ and integrating over $ \Omega_x \times \Omega_v$,
\[
\sprod{e^i}{\chi}_M = \sprod{e^i_-}{\chi_+}_M +
\sum_{m=1}^M \Big( (e^i, - \Psi^i_t )_m + ( e^i, - G( f^{h,i-1} ) \cdot \nabla \Psi^i )_m \Big) .
\]
Since $ G( f^{h,i-1}) $ is divergence free (i.e., we have a gradient field),
we may manipulate the sum above as in \eqref{DualPFormula2} and
\eqref{DualPFormula3}, ending up with
\[
\sprod{e^i}{\chi}_M = \sprod{e^i_-}{\chi_+}_M + \sum_{m=1}^M \Big(\sprod{e^i_+}{\Psi^i_+}_{m-1} - \sprod{e^i_-}{\Psi^i_-}_m
+ ( e^i_t + G( f^{h,i-1} ) \cdot \nabla e^i , \Psi^i )_m \Big).
\]
Adding and subtracting appropriate auxiliary terms, see 
\eqref{DualPFormula4}, 
this simplifies to 
\[
\sprod{e^i}{\chi}_M = \sum_{m=1}^M - \sprod{ [ f^{h,i} ]}{\Psi^i_+}_{m-1} - \Big( f^{h,i}_t + G( f^{h,i-1} ) \cdot \nabla f^{h,i} , \Psi^i \Big)_m,  
\]
where we have used \eqref{Exact_Vlasov1}. 
Let now $ \tilde{\Psi}^i $ be an interpolant of $ \Psi^i $ and use \eqref{sdvlasov} with $ g=\tilde{\Psi}^i $ to get
\begin{equation} \label{VlaFormula4}
\begin{split}
\sprod{e^i}{\chi}_M  = \sum_{m=1}^M \Big\{ &\sprod{ [ f^{h,i} ]}{\tilde{\Psi}^i_+ - \Psi^i_+}_{m-1} \\
&+  \Big( f^{h,i}_t + G( f^{h,i-1} ) \cdot \nabla f^{h,i}, 
\tilde{\Psi}^i - \Psi^i\Big)\\
&+  \Big( f^{h,i}_t + G( f^{h,i-1} ) \cdot \nabla f^{h,i},  \delta ( \tilde{\Psi}^i_t + G( f^{h,i-1} ) \cdot \nabla \tilde{\Psi}^i ) \Big)_m \Big\}.
\end{split}  
\end{equation}
In the sequel we shall use the residuals
\[
R_1^i = f^{h,i}_t + G( f^{h,i-1} ) \cdot \nabla f^{h,i} 
\]
and
\[
R^i_2 |_{S_m} = \big( f^{h,i}_+ (t_m, x,v ) - f^{h,i}_- (t_m, x,v ) \big) / h ,
\]
where $ R_2^i $ is constant in time on each slab.

Finally, we introduce 
the projections $ P $ and $ \pi $  defined in a similar way as 
in the Maxwell part, where the local projections
\[
P_m : L_2 (S_m) \rightarrow V_m^h = 
\{ u|_{S_m} ; u \in V^h \}
\]
and
\[
\pi_m :  L_2 (S_m) \rightarrow \Pi_{0,m} = \{ u \in  L_2 (S_m) ; \,
u ( \cdot , x, v) \,\, \textnormal{is constant on} \,\, I_m, (x,v) \in \Omega \} 
\]
are defined such that
$$
\int_{\Omega}(P_m\varphi)^T\cdot u\, dx dv=
\int_{\Omega}\varphi^T\cdot u\, dx dv ,\qquad \forall u\in V_m^h 
$$
and 
$$
\pi_mu|_{S_m}=\frac{1}{h} \int_{I_m} u(t, \cdot, \cdot)\, dt .
$$

The main result of this section is as follows:
\begin{theorem}\label{thenumerror}
There exists a constant $C$ such that 
\[
\norm{e^i(T, \cdot)}_{H^{-1} (\Omega)} \leq C 
\left( \norm{h R^i_1}_{L_2 (Q_T)} 
(2 + \norm{G( f^{h,i-1} )}_{L_\infty (Q_T)} ) + 
\norm{h R^i_2}_{L_2 (Q_T)} \right) .
\]
\end{theorem} 
\begin{proof}
We choose the interpolants so that 
$ \tilde{\Psi}^i = P \pi \Psi^i = \pi P \Psi^i $, then
\begin{equation*}
\begin{split} 
\sum_{m=1}^M&
\langle [f^{h,i}], \Psi^i_+ - \tilde\Psi^i_+ \rangle_{m-1} = 
\sum_{m=1}^M \langle h\frac{[f^{h,i}]} h, \Psi^i_+ - P \Psi^i_+ + P \Psi^i_+ - \tilde\Psi^i_+ \rangle_{m-1} \\
&= \sum_{m=1}^M \langle h \frac{[f^{h,i}]} h, \Psi^i_+ - P \Psi^i_+ \rangle_{m-1}
+\sum_{m=1}^M \langle h \frac{[f^{h,i}]} h , P \Psi^i_+ - \tilde\Psi^i_+\rangle_{m-1}
:= \tilde{J}_1+ \tilde{J}_2.
\end{split}
\end{equation*}
The terms $ \tilde{J}_1 $ and $ \tilde{J}_2 $ are estimated in a similar way as 
$ J_1 $ and $ J_2 $, ending up with the estimate
\[
| \tilde{J}_1 | + | \tilde{J}_2 | \leq C \norm{h R^i_2}_{L_2 (Q_T)} \norm{\Psi^i}_{H^1 (Q_T)} .
\]
Hence, it remains to bound the second and third terms in 
\eqref{VlaFormula4}. To proceed, recalling the residuals $R_1^i$ and  
$R_2^i$, we have that 
\[
\begin{split}
\Big| \sum_{m=1}^M \Big( f^{h,i}_t + G( f^{h,i-1} ) 
\cdot \nabla f^{h,i}, \tilde{\Psi}^i - \Psi^i 
+ \delta ( \tilde{\Psi}^i_t + G( f^{h,i-1} ) 
\cdot \nabla \tilde{\Psi}^i ) \Big)_m \Big| \\ 
\leq \sum_{m=1}^M | (R_1^i, \tilde{\Psi}^i - \Psi^i )_m | 
+ \delta | (R_1^i , \tilde{\Psi}_t^i + G( f^{h,i-1} ) 
\cdot \nabla \tilde{\Psi}^i )_m | \\
\leq C \norm{h R^i_1}_{L_2 (Q_T)} 
\norm{\Psi}_{H^1 (Q_T)} ( 2 + \norm{ G( f^{h,i-1} )}_{L_\infty (Q_T)} ), 
\end{split}
\]
where we used that $\delta = C h$. Summing up we have the estimate
\[
\sprod{e^i}{\chi}_M \leq C \norm{\Psi}_{H^1 (Q_T)} \left( \norm{h R^i_1}_{L_2 (Q_T)} (2 + \norm{G( f^{h,i-1} )}_{L_\infty (Q_T)} ) + \norm{h R^i_2}_{L_2 (Q_T)} \right) .
\]
Together with the following stability estimate this completes the proof.
\end{proof}

\begin{lemma}
There exists a constant $ C $ such that
\[
\norm{\Psi}_{H^1 (Q_T)} \leq C \norm{\chi}_{H^1 (\Omega)}.
\]
\end{lemma}

\begin{proof}
We start estimating the $ L_2 $-norm:
multiply the dual equation \eqref{Dual_vlasov} by $ \Psi^i $ 
and integrate over $ \Omega $ to get
\[
- \int_{\Omega} \Psi^i \d_t \Psi^i dx dv - \int_{\Omega} G( f^{h,i-1} ) \cdot \nabla \Psi^i \Psi^i dx dv = 0.
\]
The second integral is zero, since $ \Psi^i $ vanishes at the boundary of $ \Omega $, so that we have
\[
- \d_t \norm{\Psi^i}^2_{L_2 (\Omega)} = 0 .
\]
Integrating  over $ (t, T) $ yields
\[
\norm{\Psi^i (t, \cdot, \cdot) }^2_{L_2 (\Omega)} \leq \norm{\chi}_{L_2 (\Omega)}^2.
\]
Once again integrating in time, we end up with
\[
\norm{\Psi^i }_{L_2 (Q_T)} \leq \sqrt{T} \norm{\chi}_{L_2 (\Omega)}.
\]
It remains to estimate $ \norm{\nabla \Psi^i}_{L_2(Q_T)} $. To this approach we rely on
the characteristic representation of the solution for \eqref{Dual_vlasov}, see e.g., \cite{Rein:90}:
\[
\Psi^i (t, x, v) = \chi ( X(0, T-t, x, v), V(0, T-t, x, v)),
\]
with $ X(s, t, x, v) $ and $ V(s, t, x, v) $ being the solutions to the characteristic system
\begin{equation} \label{char_eq}
\begin{split}
\frac{dX}{ds} = & \hat{V}_1, \quad  \hskip 3.84cm X(s, s, x, v) = x, \\
\frac{dV}{ds} = & E^{h, i} (s, X) + B^{h, i} (s, X) M \hat{V}, \quad V(s, s, x, v) = v,
\end{split} 
\end{equation}
where
\[
M = \left( \begin{array}{cccc}
0 & 1  \\
-1 & 0  \\
\end{array} \right) .
\]
Hence, we have
\[
\nabla \Psi^i = \nabla X (0, T-t, x, v) \frac{\d \chi}{\d x} + \nabla V_1 (0, T-t, x, v) \frac{\d \chi}{\d v_1} + \nabla V_2 (0, T-t, x, v) \frac{\d \chi}{\d v_2} .
\]
Thus, it suffices to estimate the gradients of $ X $ and $ V $. 
Below we shall estimate the derivatives of $ X $, $ V_1 $ and $ V_2 $
with respect to $ x $. The estimates with respect to $ v_i, \,\, i = 1, 2 $ are done in a similar way.
Differentiating \eqref{char_eq} with respect to $x$ we get
\[
\begin{split}
\frac{d}{ds} \d_{x} X =& \frac{\d_{x} V_1}{\sqrt{1 + V_1^2 + V_2^2}} 
- \frac{V_1 (V_1 \d_{x} V_1 + V_2 \, \d_{x} V_2)}{(1 + V_1^2 + V_2^2)^{3/2}} \\
\frac{d}{ds} \d_{x} V_1 =& \d_x E_1^{h, i} \d_x X + \frac{V_2 \d_x B^{h, i} \d_x X}{\sqrt{1 + V_1^2 + V_2^2}}\\
&\qquad\qquad
+ B^{h, i} \left( \frac{\d_{x} V_2}{\sqrt{1 + V_1^2 + V_2^2}}
- \frac{V_2 (V_1 \d_{x} V_1 + V_2 \, \d_{x} V_2)}{(1 + V_1^2 + V_2^2)^{3/2}} \right) \\
\frac{d}{ds} \d_{x} V_2 =& \d_x E_2^{h, i} \d_x X - \frac{V_1 \d_x B^{h, i} \d_x X}{\sqrt{1 + V_1^2 + V_2^2}}\\
&\qquad\qquad - B^{h, i} \left( \frac{\d_{x} V_1}{\sqrt{1 + V_1^2 + V_2^2}}
- \frac{V_1 (V_1 \d_{x} V_1 + V_2 \, \d_{x} V_2)}{(1 + V_1^2 + V_2^2)^{3/2}} \right).  
\\
\end{split}
\]
Integrating these equations over $ [s, t] $ 
and then taking the absolute values give 
\begin{equation*}
\begin{split}
\abs{\partial_x X(s)}& \le 1+\int_s^t \Big(2\abs{\partial_x V_1}+\abs{\partial_x V_2}\Big)\, d\tau \\ 
\abs{\partial_x V_1(s)}&\le \int_s^t\Big(\norm{\partial_x E_1^h}_\infty+\norm{\partial_xB^h}_\infty \Big) 
\abs{\partial_x X} + 
\norm{B^h}_\infty\Big(\abs{\partial_x V_1}+ 2\abs{\partial_x V_2}\Big)\, d\tau \\
\abs{\partial_x V_2(s)}&\le \int_s^t\Big(\norm{\partial_x E_2^h}_\infty+\norm{\partial_xB^h}_\infty \Big) 
\abs{\partial_x X} +
\norm{B^h}_\infty\Big(2\abs{\partial_x V_1}+ \abs{\partial_x V_2}\Big)\, d\tau.
\end{split}
\end{equation*}
Summing up we have that
\begin{equation*}
\begin{split}
\abs{\partial_x X(s)}+\abs{\partial_x V_1(s)}+\abs{\partial_x V_2(s)}\le &
1+\int_s^t 2\norm{\partial_x W^h}_\infty \abs{\partial_x X} \\
&+
\Big(2+3\norm{B^h}_\infty\Big)\Big(\abs{\partial_x V_1}+\abs{\partial_x V_2}\Big)\, d\tau .
\end{split} 
\end{equation*} 
Now an application of the Gr\"onwall's lemma yields 
\begin{equation}\label{Gronwall1}
\abs{\partial_x X(s)}+\abs{\partial_x V_1(s)}+\abs{\partial_x V_2(s)}\le 
\exp\Big( \int_s^t 2+2 \norm{\partial_x W^h}_\infty +3 \norm{B^h}_\infty\, d\tau \Big). 
\end{equation}
By similar estimates for derivatives with respect to velocity components 
we have 
\begin{equation}\label{Gronwall2} 
\abs{\partial_{v_j}X(s)}+\abs{\partial_{v_j}V_1(s)}+\abs{\partial_{v_j}V_2(s)}
\le \exp\Big(\int_s^t 2 + 3\norm{B^h}_\infty\, d\tau\Big), \quad j=1, 2. 
\end{equation}
The estimates \eqref{Gronwall1} and \eqref{Gronwall2} would result to the 
key inequalities
\begin{equation}\label{ESTfixv1}
\norm{\partial_x\Psi}\le C_T\norm{\nabla \chi}\qquad\mbox{and}\qquad 
\norm{\partial_{v_j}\Psi}\le C_T\norm{\nabla \chi}, \quad j=1,2, 
\end{equation}
which together with the equation for $ \Psi $ gives 
\begin{equation}\label{ESTfit1}
\norm{\partial_t\Psi}\le C_T\norm{\nabla\chi}.
\end{equation}
Summing up we have shown that 
\begin{equation}\label{Gradestim1}
\norm{\Psi}_{H^1(Q_T)}\le C_T\norm{\chi}_{H^1(\Omega)},
\end{equation}
which proves the desired result.
\end{proof}

Using the assumption that the iteration error $ \E^i $ converges
and is dominated by the numerical error $ e^i $,
together with Theorem \ref{thenumerror}, we get the following result:
\begin{corollary}
There exists a constant $ C $ such that
\[
\begin{split}
\norm{f (T, \cdot) - f^{h,i} (T, \cdot)}_{H^{-1}(\Omega)} & \leq \\
 C\Big( \norm{h R^i_1}_{L_2 (Q_T)} & (2 + \norm{G( f^{h,i-1} )}_{L_\infty (Q_T)} )
+ \norm{h R^i_2}_{L_2 (Q_T)} \Big).
\end{split}
\]
\end{corollary}

\section{Conclusions and future works} \label{sec5}

We have presented an
 a posteriori error analysis of the streamline diffusion (SD) 
scheme for the
relativistic one and one-half dimensional Vlasov-Maxwell system.
The motivation behind our choice of the method is that 
the standard finite element method for  
 hyperbolic problems is sub-optimal. 
The streamline diffusion is performed 
slab-wise and allows jump discontinuities across the time grid-points. 
The SD approach 
have stabilizing effect due to the fact that, adding a multiple
of the streaming term to the test function,
it corresponds to an automatic add
of diffusion to the equation. 

Numerical study of the VM system has some draw-backs
in both stability and convergence.
The VM system lacks dissipativity which, in general,
affects the stability. Further, 
$L_2 (L_2)$ a posteriori error bounds would only be of 
the order of the norms of residuals.
In our study, in order to derive error estimates with convergence rates
of order $ h^\alpha $, for some $ \alpha > 0 $,
the $ H^{-1} (H^{-1})$ and $L_\infty (H^{-1})$ environments
are employed. However, because of the lack of dissipativity,
the $ H^{-1} (H^{-1})$-norm is not extended to the Vlasov part,
where appropriate stability estimates are not available.
Therefore the numerical study of the Vlasov part is restricted to
the $L_\infty (H^{-1})$ environment.

The computational aspects and implementations, 
which justify the theoretical results of this part,  
 are the subject of a forthcoming study which is addressed in \cite{Bondestam_Standar:2017}. 

Future studies, in addition to considering higher dimensions and 
implementations, may contain investigations concerning 
the assumption on the convergence of the iteration procedure, 
see end of Section 2. 
 
We also plan to extend this study to Vlasov-Schr\"{o}dinger-Poisson system, 
where we rely on 
the theory developed by Ben Abdallah et al.\ in 
\cite{Ben:2006} and \cite{Ben:2004} 
and consider a novel discretization procedure 
based on the mixed virtual element method, as in Brezzi et al.\
 \cite{Brezzi_Falk:2014}.

\def\listing#1#2#3{{\sc #1}:\ {\it #2},\ #3.}

\end{document}